\documentclass{article}

\usepackage{amsmath, amssymb}
\usepackage{amsthm}
\usepackage{url}
\usepackage{amsthm,amsfonts,amssymb,amsmath}

\usepackage{xparse}   
\NewDocumentCommand{\ceil}{s O{} m}{%
	\IfBooleanTF{#1} 
	{\left\lceil#3\right\rceil} 
	{#2\lceil#3#2\rceil} 
}

\NewDocumentCommand{\floor}{s O{} m}{%
	\IfBooleanTF{#1} 
	{\left\lfloor#3\right\rfloor} 
	{#2\lfloor#3#2\rfloor} 
}

\newcommand{\Mod}[1]{\ (\mathrm{mod}\ #1)}

\newtheoremstyle{dot}{}{}{\itshape}{}{\bfseries}{}{ }{}

\newtheorem{thm}{Theorem}[section]

\newtheorem{lem}[thm]{Lemma}
\theoremstyle{definition}
\newtheorem{defi}[thm]{Definition}

\newtheorem{remark}[thm]{Remark}

\numberwithin{equation}{section}

\begin{document}

\begin{center}
{\Large\textbf{On Exactly $3$-Deficient-Perfect Numbers}}
\end{center}
\vspace{0.5cm}

\begin{center}
Saralee Aursukaree and Prapanpong Pongsriiam$^*$
\vspace{0.25cm}

Department of Mathematics, Faculty of Science, Silpakorn University, Nakhon Pathom, 73000, Thailand
\vspace{0.25cm}

aursukaree.s@gmail.com, prapanpong@gmail.com, pongsriiam\_p@silpakorn.edu

\end{center}
\vspace{1cm}

\begin{abstract}
	Let $n$ and $k$ be positive integers and $\sigma(n)$ the sum of all positive divisors of $n$. We call $n$ an exactly $k$-deficient-perfect number with deficient divisors $d_1, d_2, \ldots, d_k$ if $d_1, d_2, \ldots, d_k$ are distinct proper divisors of $n$ and $\sigma (n)=2n-(d_1+d_2+\ldots + d_k)$.  In this article, we show that the only odd exactly $3$-deficient-perfect number with at most two distinct prime factors is $1521=3^2 \cdot 13^2$.
\end{abstract}


keywords: perfect number, deficient-perfect number, near-perfect number, the sum of divisor function, divisor

\section{Introduction}
Throughout this article, let $n$ be a positive integer, $\sigma (n)$ the sum of all positive divisors of $n$, and $\omega (n)$ the number of distinct prime factors of $n$. We say that $n$ is perfect if $\sigma(n)=2n$. It is well-known that $n$ is even and perfect if and only if $n=2^{p-1}\left(2^p -1\right)$ where $p$ and $2^p-1$ are primes. It has also been a long standing conjecture that there are infinitely many even perfect numbers and that an odd perfect number does not exist. Attempting to understand perfect numbers, mathematicians have studied other closedly related concepts. Recall that if $\sigma(n)<2n$, then $n$ is said to be deficient; if $\sigma(n) >2n$, then $n$ is abundant; if $\sigma (n)=2n+1$, then $n$ is quasiperfect; if $\sigma (n)=2n-1$, then $n$ is almost perfect. For more information on this topic, see for example in the work of Cohen \cite{C80,Co}, Hagis and Cohen \cite{Ha}, Kishore \cite{Kis}, Ochem and Rao \cite{OcRa}, Yamada \cite{Yam}, and in the online databases GIMPS \cite{GIMPS} and OEIS \cite{Slo}.

Sierpi\'nski \cite{Si} called $n$ pseudoperfect if $n$ can be written as a sum of some of its proper divisors. Pollack and Shevelev \cite{Pol} have recently initiated the study of a subclass of pseudoperfect numbers leading to an active investigation. We summarize it in the following definition.

\begin{defi}
	Let $n$ and $k$ be positive integers. We say that $n$ is near-perfect if $n$ is the sum of all of its proper divisors except one of them. In addition, $n$ is $k$-near-perfect if $n$ can be written as a sum of all of its proper divisors with at most $k$ exceptions. Moreover, $n$ is exactly $k$-near-perfect if $n$ is expressible as a sum of all of its proper divisors with exactly $k$ exceptions. The exceptional divisors are said to be redundant. In other words,
	\begin{align*}
	&	\text{$n$ is near-perfect with a redundant divisor $d$} \Leftrightarrow \text{$1\leq d<n $, $d\mid n $, and $\sigma(n)=2n+d$;}\\
	&	\text{$n$ is $1$-near-perfect} \Leftrightarrow \text{$n$ is perfect or $n$ is near-perfect;}\\
	&	\text{$n$ is exactly $k$-near-perfect with redundant divisors $d_1, d_2, \ldots, d_k$} \Leftrightarrow \\
	&\text{ $d_1,d_2,\ldots,d_k$ are distinct proper divisors of $n$ and $\sigma(n) = 2n+d_1+d_2+\cdots+d_k$.}
	\end{align*}
\end{defi}
Motivated by the concept of near-perfect numbers, Tang, Ren, and Li \cite{TRL} define the notion of deficient-perfect numbers which also leads to an interesting research problem.
\begin{defi}
	Let $n, k \in \mathbb{N}$. Then $n$ is called a deficient-perfect number with a deficient divisor $d$ if $d$ is a proper divisor of $n$ and $\sigma(n)=2n-d$. Furthermore, $n$ is exactly $k$-deficient-perfect with deficient divisors $d_1, d_2, \ldots, d_k$ if $d_1, d_2, \ldots, d_k$ are distinct proper divisors of $n$ and  $\sigma(n)=2n-(d_1+d_2+\cdots+d_k)$. In addition, $n$ is $k$-deficient-perfect if $n$ is perfect or $n$ is exactly $\ell$-deficient-perfect for some $\ell = 1, 2, \ldots, k$.
\end{defi}
In 2012, Pollack and Shevelev \cite{Pol} showed that the number of near-perfect numbers not exceeding $x$ is $\ll x^{5/6+o(1)}$ as $x \rightarrow \infty $, and that if $k$ is fixed and is large enough, then there are infinitely many exactly $k$-near-perfect numbers. A year later, Ren and Chen \cite{RenC} determined all near-perfect numbers $n$ which have $\omega (n)=2$ and we can see from this classification that all such $n$ are even. In the same year, Tang, Ren, and Li \cite{TRL} proved that there is no odd near-perfect number $n$ with $\omega (n)=3$ and found all deficient-perfect numbers $m$ with $\omega (m) \leq 2$. After that, Tang and Feng \cite{TF} extended it by showing that there is no odd deficient-perfect number $n$ with $\omega (n)=3$. Tang, Ma, and Feng \cite{TMF} obtained in 2016 the only odd near-perfect number with $\omega (n)=4$, namely, $n=3^4\cdot 7^2 \cdot 11^2 \cdot 19^2$, while Sun and He \cite{SunHe} asserted in 2019 that the only odd deficient-perfect number $n$ with $\omega (n)=4$ is $n=3^2 \cdot 7^2\cdot 11^2 \cdot 13^2$. Cohen et.al \cite{Cohen} have recently improved the estimate of Pollack and Shevelev \cite{Pol} on the number of near-perfect numbers $\leq x$. Hence, most results in the literature are devoted to characterizing, only when $k=1$, the exactly $k$-near-perfect or exactly $k$-deficient-perfect numbers. Chen \cite{Chen} started a slightly new direction by determining all $2$-deficient-perfect numbers $n$ with $\omega (n) \leq 2$.

In this article, we continue the investigation on odd $3$-deficient-perfect numbers $n$ with $\omega (n) \leq 2$. We obtain that the only such $n$ is $n=1521=3^2 \cdot 13^2$. For other articles related to the divisor functions or divisibility problems, see for example in \cite{BGLLHT, BLMS, CCZ, Cub, DPT, JPo, Khan, LQu, LSZ, LYo, OPo, PPo, Pon19.1, Pon19.2, Pon14.1, PVa, PVa2, SPa, Ste, Yam}.

\section{Main Results}
By the definition, $n$ is deficient-perfect if and only if $n$ is exactly $1$-deficient-perfect. Tang and Feng \cite[Lemma 2.1]{TF} show that if $n$ is deficient-perfect and $n$ is odd, then $n$ is a square. We can extend their result to the following form.

\begin{lem} \label{lemma1}
	Let $n$ and $k$ be positive integers. Suppose that  $n$ is exactly  $k$-deficient-perfect and $n$ is odd. Then $n$ is a square if and only if $k$ is odd. In particular, if $n$ is odd and exactly $3$-deficient-perfect, then $n$ is a square.
\end{lem}
\begin{proof}
	Since $1$ has no proper divisor, we can assume that $n>1$ and write $n=p_1^{\alpha_1}p_2^{\alpha_2}\cdots p_r^{\alpha_r}$ where $p_1, \ldots, p_r$ are distinct odd primes and $\alpha_1, \alpha_2, \ldots, \alpha_r$ are positive integers. Let $d_1, d_2, \ldots, d_k$  be distinct proper divisors of $n$ such that
	\begin{equation}
		2n-d_1-d_2-\cdots-d_k=\sigma(n)=\prod_{i=1} ^r \sigma(p_i^{\alpha_i}) = \prod _{i=1} ^r (1+p_i+\cdots+p_i^{\alpha_i}). \label{lem1}
	\end{equation} 
	Since $n$ is odd, $d_i$ and $p_j$ are odd for every $i=1, 2, \ldots, k$ and $j= 1, 2, \ldots, r$. Reducing \eqref{lem1} $\bmod$ $2$, we obtain $k \equiv \prod _{i=1} ^r (\alpha_i+1) \Mod{2}$. From this, we have the equivalence $k$ is odd $\Leftrightarrow$ $\alpha_i$ is even for all $i$ $\Leftrightarrow$ $n$ is a square, which proves our lemma.
\end{proof}

Tang, Ren, and Li \cite{TRL} determine all deficient-perfect numbers $n$ with $\omega (n)\leq 2$. In particular, they show that if $\omega (n)=1$ and $n$ is deficient-perfect, then $n$ is a power of $2$. We can extend this for exactly $k$-deficient-perfect numbers as follows.

\begin{lem}
	Let $n\geq 2$, $k \geq 1$ be integers. If $n$ is exactly $k$-deficient-perfect and $\omega(n)=1$, then $k=1$ and $n$ is a power of $2$. Consequently, if $n$ is exactly $k$-deficient-perfect and $k\geq2$,  then $n$ has at least two distinct prime divisors. In particular, every exactly $3$-deficient-perfect number $n$ has $\omega (n) \geq 2$.
\end{lem}
\begin{proof}
	Suppose $n=p^\alpha$ and the deficient divisors of $n$ are $d_1=p^{\beta_1}$, $d_2=p^{\beta_2}, \ldots$, $d_k=p^{\beta_k}$, where $\alpha >\beta _1 >\beta _2 > \cdots > \beta _k \geq 0 $. Since $(p^{\alpha+1}-1)/(p-1)=\sigma(n)=2n-d_1-\cdots-d_k$, we obtain
	\begin{equation}
	(d_1+d_2+\cdots+d_k)(p-1)-1=p^\alpha(p-2). \label{lem2}
	\end{equation}
	If $p\geq 3$, then 
	\begin{align*}
	p^\alpha &\leq p^{\alpha }(p-2) =(d_1+d_2+\cdots+d_k)(p-1)-1\\
	&\leq (p^{\alpha-1}+p^{\alpha-2}+\cdots +p^{\alpha-k})(p-1)-1 = p^\alpha -p^{\alpha-k}-1,
	\end{align*}
	which is impossible. Therefore $p=2$ and $n$ is a power of $2$. By \eqref{lem2}, we also obtain, $d_1+\cdots+d_k=1$, which implies $k=1$ and $\beta_1=0$.
\end{proof}

We now give the main result of this paper.
\begin{thm}
	The only odd exactly $3$-deficient-perfect number which has $\omega (n)=2$ is $1521=3^2\cdot 13^2$ with three deficient divisors $d_1=507$, $d_2=117$, and $d_3=39$.
\end{thm}

\begin{proof} It is easy to check that if $n=1521$ and $d_1$, $d_2$, $d_3$ are as above, then $\omega (n)=2$, $n$ is odd, $d_1$, $d_2$, $d_3$ are proper divisors of $n$, $\sigma (n)=2n-d_1-d_2-d_3$, and so $n$ is exactly $3$-deficient-perfect. For the other direction, assume that $n$ is odd, $\omega (n)=2$, and $n$ is exactly $3$-deficient-perfect. By Lemma \ref{lemma1}, $n$ is a square, so we can write $n=p_1^{2\alpha} p_2^{2\beta}$ where $2<p_1<p_2$ and $\alpha, \beta \geq1$. In addition, let $d_1>d_2>d_3$ be the deficient divisors of $n$, and let $D_1=n/d_1$, $D_2=n/d_2$, $D_3=n/d_3$. Then $p_1 \leq D_1<D_2<D_3 \leq n$. Since $\sigma(n)=2n-d_1-d_2-d_3$, we obtain
	\begin{align}
		2 &=  \frac{\sigma (n)}{n} + \frac{d_1}{n} +  \frac{d_2}{n} +  \frac{d_3}{n}\nonumber \\
		&=  \frac{({p_1}^{2\alpha +1}-1) ({p_2}^{2\beta +1}-1) }{(p_1-1)(p_2-1)p_1^{2\alpha} p_2^{2\beta}} + 	\frac{1}{D_1}+ \frac{1}{D_2}+ \frac{1}{D_3} \nonumber\\
		&<  \frac{p_1p_2 }{(p_1-1)(p_2-1)} + \frac{1}{D_1}+ \frac{1}{D_2}+ \frac{1}{D_3}. \label{first}
	\end{align}
	If $p_1 \geq 5$, then $p_1/(p_1-1)\leq5/4$, $p_2 \geq7$, $p_2/(p_2-1)\leq7/6$, $D_1\geq5$, $D_2\geq7$, $D_3\geq25$, and \eqref{first} implies that
	$$
	2< \frac{5}{4}\cdot  \frac{7}{6}+ \frac{1}{5}+ \frac{1}{7}+ \frac{1}{25}=1.8411\ldots,
	$$
	which is a contradiction.
	So $p_1=3$. For convenience, let $p_2=p$. Then $n=3^{2\alpha}p^{2\beta}$ and \eqref{first} becomes 
	\begin{equation} 
	2< \frac{3p }{2(p-1)} + \frac{1}{D_1}+ \frac{1}{D_2}+ \frac{1}{D_3}.\label{2leq}
	\end{equation}
	If $p\geq83$, then \eqref{2leq} leads to 
	$2<(3/2)(83/82)+1/3+1/9+1/27=1.9997\ldots$, which is impossible. So $5\leq p \leq 79$. Recall that the primes in $[5,79]$ are $5$, $7$, $11$, $13$, $17$, $19$, $23$, $29$, $31$, $37$, $41$, $43$, $47$, $53$, $59$, $61$, $67$, $71$, $73$, $79$. If $p\geq 11$ and $D_1> 3$, then $D_1\geq 9$, $D_2\geq11$, $D_3\geq 27$, and \eqref{2leq} gives $2<(3/2)(11/10)+1/9+1/11+1/27=1.8890\ldots$, which is false. Therefore
	\begin{equation}
	\text{if $p\geq 11$, then $D_1=3$.} \label{B}
	\end{equation}
	Similarly, if $p\geq 23$ and $D_2>9$, then $2<(3/2)(23/22)+1/3+1/23+1/27=1.9820\ldots$, which is not true. Thus 
	\begin{equation}
	\text{if $p\geq 23$, then $D_2=9$.} \label{C}
	\end{equation}
	Next, we divide our calculations into eleven cases according to the value of $p$. In addition, we write the possible values of $D_1$, $D_2$, $D_3$ in an increasing order.
	
	\indent\textbf{Case 1} $47\leq p\leq79$. By \eqref{B} and \eqref{C}, we have $D_1=3$, $D_2=9$, and the possible values of $D_3$ in an increasing order are $D_3 = 27, p, 81, \ldots $.
	If $D_3\geq p$, then \eqref{2leq} implies
	$2<(3/2)(47/46)+1/3+1/9+1/47=1.9983\ldots$, which is false.
	So $D_3=27$. Then $2\alpha \geq 3$, $d_1=n/D_1=3^{2\alpha-1}p^{2\beta}$, $d_2=3^{2\alpha-2}p^{2\beta}$, $d_3=3^{2\alpha-3}p^{2\beta}$, and 
	\begin{align*}
		\frac{\left( 3^{2\alpha+1}-1 \right) \left( p^{2\beta+1}-1 \right)}{ 2(p-1) } &=\sigma(3^{2\alpha}p^{2\beta})= 2\cdot3^{2\alpha}p^{2\beta}-d_1-d_2-d_3\\
		&=3^{2\alpha-3}p^{2\beta}(2\cdot3^3-3^2-3-1) 
		=41\cdot3^{2\alpha-3}p^{2\beta}.
	\end{align*}
	This leads to
	\begin{equation}
	3^{2\alpha -3}= \frac{{p}^{2\beta +1 }-1}{(82-p)p^{2\beta}-81}. \label{3927} 
	\end{equation}
	The left-hand side of \eqref{3927} is an integer, and we get a contradiction by showing that the right-hand side of \eqref{3927} is not an integer. From this point on, let $A$ be the number on the right-hand side of \eqref{3927}. If $p=47$, then $A$ is equal to 
	$$
	\frac{47\cdot 47^{2\beta}-1}{35\cdot 47^{2\beta}-81}=1+ \frac{12\cdot47^{2\beta}+80}{35\cdot 47^{2\beta}-81}
	=1+ \frac{12+(80/47^{2\beta})}{35-(81/47^{2\beta})} \in (1, 2),
	$$
	and so $A \notin \mathbb{Z}$. Similarly, 
	\begin{align*}
		&\text{if $p=53$, then $A=1+\frac{24p^{2\beta}+80}{29 p^{2\beta}-8} \in (1,2)$;} \\ 
		&\text{if $p=59$, then $A=2+\frac{ 13p^{2\beta}+161 }{ 23 p^{2\beta}-81 } \in (2,3)$;} \\
		&\text{if $p=61$, then $A=2+\frac{ 19p^{2\beta}+161 }{ 21p^{2\beta}-81 } \in (2,3)$;} \\
		&\text{if $p=67$, then $A=4+\frac{7p^{2\beta}+323} { 15p^{2\beta}-81} \in (4,5)$.}
	\end{align*}
	The remaining cases $p=71, 73, 79$ lead to $A\in(6,7)$, $A\in(8,9)$, and $A\in(26,27)$, respectively. In any case, $A\notin \mathbb{Z}$ and we have a contradiction. Hence this case does not lead to a solution.
	
	\indent\textbf{Case 2} $p \in \{37,41,43\}$. By \eqref{B} and \eqref{C}, we have $D_1=3$, $D_2=9$, and $D_3 = 27, p, 81, \ldots $. If $D_3\geq 81$, then \eqref{2leq} implies $2<(3/2)(37/36)+1/3+1/9+1/81=1.9984\ldots$, which is not possible. So $D_3=\{27,p\} $.
	
	\indent\textbf{Case 2.1} $D_1=3$, $D_2=9$, and $D_3=27$. Then $2\alpha \geq3$, \eqref{3927} holds, and the calculations in Case 1 work in this case too. Since \eqref{3927} holds, we still let $A$ be the right-hand side of \eqref{3927}. Therefore if $p=37$, then $A\in(0,1)$ and if $p\in\{41,43\}$, then $A\in(1,2)$, which is a contradiction.
	
	\indent\textbf{Case 2.2} $D_1=3$, $D_2=9$, and $D_3=p$. Then 
	\begin{align*}
	\frac{\left( 3^{2\alpha+1}-1\right) \left( p^{2\beta+1}-1\right)}{2(p-1)} &=\sigma(3^{2\alpha}p^{2\beta})=\sigma(n)=2n-d_1-d_2-d_3
	\\
	&=2\cdot 3^{2\alpha}p^{2\beta}-3^{2\alpha-1}p^{2\beta}-3^{2\alpha-2}p^{2\beta}-3^{2\alpha}p^{2\beta -1}\\
	&=3^{2\alpha-2}p^{2\beta-1} (14p-9),
	\end{align*}
	which implies
	\begin{equation}
		3^{2\alpha-2}= \frac{p^{2\beta+1}-1}{(46p-p^2-18)p^{2\beta-1}-27} \label{case22}.
	\end{equation}
	The equality \eqref{case22} can be used in the same way as \eqref{3927}. So let $B$ be the number on the right-hand side of \eqref{case22}. Similar to the previous computation, we see that if $p=37$, then $B\in(4,5)$ and if $p=43$, then $B\in (16,17)$ which contradicts the fact that $B=3^{2\alpha-2} \in \mathbb{Z}$. Suppose $p=41$. Then $B\in (8,10)$, which implies $B=9$. Equating the right-hand side of \eqref{case22} with $B=9$, substituting $p=41$, and performing a straightforward manipulation leads to $41^{2\beta-1}=121$, which is not possible. Hence there is no solution in this case.
	
	\begin{remark}
		Before going further, we note that the calculations similar to \eqref{3927} and \eqref{case22} and their applications occur throughout the proof, and we give less details than those in \eqref{3927} and \eqref{case22}.
	\end{remark} 
	
	\indent\textbf{Case 3} $p\in \{29,31\}$. Then by \eqref{B} and \eqref{C}, $D_1=3$, $D_2=9$, and $D_3 =27$, $p$, $81$, $3p$, $243$, $9p$, $729$, $\ldots $. If $p=31$ and $D_3\geq 243$, then \eqref{2leq} implies $2<(3/2)(31/30)+1/3+1/9+1/243=1.9985\ldots$, which is false. Similarly, assuming $p=29$ and $D_3\geq 729$ leads to a false inequality. Therefore 
	\begin{align}
		&\text{if $p=31$, then $D_3\in \{27,31,81,93\}$}, \label{D} \\ 
		&\text{if $p=29$, then $D_3\in \{27,29,81,87,243,261\}$.} \label{E}
	\end{align}
	Next, we divide our calculations according to the value of $D_3$.
	
	\indent\textbf{Case 3.1} $D_3=27$. Then \eqref{3927} holds and the same method still works. We obtain
	\begin{align*}
		\text{if $p=29$, then $A=\left( 29p^{2\beta} -1\right)/\left( 53p^{2\beta}-81\right) \in (0,1)$};\\
		\text{if $p=31$, then $A=\left( 31p^{2\beta} -1\right)/\left( 51p^{2\beta}-81\right) \in (0,1)$}. 
	\end{align*}
	So $A\notin \mathbb{Z}$ and we get a contradiction.
	
	\indent\textbf{Case 3.2} $D_3=p\in\{29,31\}$. Then \eqref{case22} holds and 
	\begin{align*}
		\text{if $p=29$, then  $B=\left( 841p^{2\beta-1} -1\right)/\left( 475p^{2\beta-1}-27\right) \in (1,2)$};\\
		\text{if $p=31$, then $B=\left( 961p^{2\beta-1} -1\right)/\left( 447p^{2\beta-1}-27\right) \in (1,2)$},
	\end{align*}
	which is a contradiction.
	
	\indent\textbf{Case 3.3} $D_3=81$. Similar to the calculations for \eqref{3927} and \eqref{case22}, we write $\sigma (n)=2n-d_1-d_2-d_3$ where $d_1$, $d_2$ are the same as before, but $d_3=n/D_3=3^{2\alpha-4}p^{2\beta}$ and $2\alpha \geq 4$. After a similar algebraic manipulation, we get
	\begin{equation}
		3^{2\alpha -4}= \frac{p^{2\beta+1}-1}{(250-7p)p^{2\beta}-243}. \label{case33}
	\end{equation}
	When $p=29$ or $31$, the right-hand side of \eqref{case33} is in the interval $(0,1)$, which is impossible.
	
	\indent\textbf{Case 3.4} $D_3=93$. By \eqref{D} and \eqref{E}, we know that $p=31$. Similar to Case 3.3 but with $d_3=n/D_3=3^{2\alpha-1}p^{2\beta-1}$, we start with $\sigma(n)=2n-d_1-d_2-d_3$ and perform an algebraic manipulation to obtain $$3^{2\alpha-2}= \frac{p^{2\beta+1}-1}{(34p-p^2-6)p^{2\beta-1}-27}= \frac{961p^{2\beta-1}-1}{87p^{2\beta-1}-27} \in (11,12),$$ which is false.
	
	\indent\textbf{Case 3.5} $D_3\in \{87,243,261\}$. By \eqref{D} and \eqref{E}, we have $p=29$. Similar to Case 3.3 but with different values of $d_3=n/D_3=3^{2\alpha-1}p^{2\beta-1}$, $3^{2\alpha-5}p^{2\beta}$, or $ 3^{2\alpha-2}p^{2\beta-1}$ when $D_3=87, 243,$ or $261$, respectively. These lead to 
	\begin{align*}
		&\text{$2\alpha \geq 2$ and $3^{2\alpha-2}= \frac{p^{2\beta+1}-1}{(34p-p^2-6)p^{2\beta-1}-27}= \frac{841p^{2\beta-1}-1}{139p^{2\beta-1}-27} \in (6,7)$,} \text{ if $D_3=87$;}
		\\ &\text{$2\alpha \geq 5$  and $3^{2\alpha-5}= \frac{p^{2\beta+1}-1}{(754-25p)p^{2\beta}-729}= \frac{29^{2\beta+1}-1}{29^{2\beta+1}-729} \in (1,2)$,} \text{ if $D_3=243$;}
		\\ &\text{$2\alpha \geq 2$ and
		$3^{2\alpha-2}= \frac{p^{2\beta+1}-1}{(30p-p^2-2)p^{2\beta-1}-27}= \frac{841p^{2\beta-1}-1}{27p^{2\beta-1}-27} \in (31,33)$,} \text{ if $D_3=261$}.
	\end{align*}
	In any case, we get a contradiction.
	
	\indent\textbf{Case 4} $p=23$. By \eqref{B} and \eqref{C}, we have $D_1=3$ and $D_2=9$. We start from 
	\begin{align*}
		(3^{2\alpha+1}-1)(p^{2\beta+1}-1)&=2(p-1)\sigma(n)=2(p-1)(2n-d_1-d_2-d_3)
		\\&=28(p-1)3^{2\alpha-2}p^{2\beta}-2(p-1)d_3.
	\end{align*}
	Writing $(3^{2\alpha+1}-1)(p^{2\beta+1}-1)=27p3^{2\alpha-2}p^{2\beta}-3^{2\alpha+1}-p^{2\beta+1}+1$, the above leads to
	\begin{equation}
		(28-p)3^{2\alpha-2}p^{2\beta}-3^{2\alpha+1}-p^{2\beta+1}+1+2(p-1)d_3=0. \label{case5}
	\end{equation}
	Multiplying both sides of \eqref{case5} by $28-p$ and factoring a part of it gives us 
	\begin{equation}
		\left((28-p)3^{2\alpha-2}-p\right)\left((28-p)p^{2\beta}-27\right)=28(p-1)-2(28-p)(p-1)d_3. \label{F}
	\end{equation}
	Substituting $p=23$, the equation \eqref{F} becomes 
	\begin{equation}
		(5\cdot 3^{2\alpha-2}-23)(5\cdot 23^{2\beta}-27)=616-220d_3. \label{case41}
	\end{equation}
	Let $A_1$ and $A_2$ be the expressions on the left and the right of \eqref{case41}, respectively. If $\alpha \geq2$, then $A_1>616$ while $A_2<616$, which is not the case. So $\alpha=1$ and $A_1=-18(5\cdot23^{2\beta}-27)$. Since $3 \mid A_1$ and $3\nmid 616$, we see that $3\nmid d_3$. Since $d_3 \mid n$ and $n=3^{2\alpha}23^{2\beta}$, we obtain $d_3=23^{b_3}$ for some $b_3 \geq0$. If $b_3=0$, then $A_2=616-220 \equiv 5 \Mod{23}$; if $b_3 \geq 1$, then $A_2\equiv 18 \Mod{23}$. But $A_1\equiv3 \Mod{23}$, and so $A_1=A_2$ and $A_1 \not\equiv A_2 \Mod{3}$, which is not possible.

	\indent\textbf{Case 5} $p=19$. By \eqref{B}, $D_1=3$. So $\{ D_2,D_3\} \subseteq \{9, 19, 27, 57, \ldots\}$. If $D_2\geq 19$ and $D_3 \geq 57$, then \eqref{2leq} implies that $2<(3/2)(19/18)+1/3+1/19+1/57=1.9868\ldots$, which is not true. Therefore ($D_2=9$) or ($D_2=19$ and $D_3=27$).
	
	\indent\textbf{Case 5.1} $D_2=9$. Then the computation in Case 4 still works and  \eqref{F} holds. Substituting $p=19$ in \eqref{F} and dividing both sides by $9$, we obtain
	\begin{equation}
		(3^{2\alpha}-19)(19^{2\beta}-3)=56-36d_3. \label{case51}
	\end{equation}
	Let $A_3$, $A_4$ be the expressions on the left and the right of \eqref{case51}, respectively. If $\alpha \geq 2$, then $A_3>56$ while $A_4<56$, which is not true. Therefore $\alpha=1$. Then $11\equiv A_3\equiv A_4 \equiv -1+2d_3 \Mod{19}$, and so $19\nmid d_3$. Since $d_3 \mid n$ and $n=3^{2\alpha}p^{2\beta}=3^2\cdot 19^{2\beta}$, we see that $d_3=1,3,9$. Substituting $d_3=1,3,9$ in \eqref{case51} leads to $5\cdot19^{2\beta}=5, 41, 149$, respectively, which has no solution.
	
	\indent\textbf{Case 5.2} $D_2=19$ and $D_3=27$. Similar to the calculations for \eqref{3927} and \eqref{case41} but with different values of $d_2$ and $d_3$, we obtain after an algebraic manipulation that
	$$3^{2\alpha-3}= \frac{361\cdot 19^{2\beta -1}-1}{117\cdot 19^{2\beta-1}-81} \in (3,4), $$
	which is not possible.
	
	\indent\textbf{Case 6} $p\in \{11, 13, 17\}$. Then by \eqref{B}, we have $D_1=3$. The possible values of $D_2$ and $D_3$ listed in an increasing order are $9$, $p$, $27$, $3p$, $81$, $9p$, $\min \{ p^2,243\}$, $\max \{ p^2,243\}$, $\ldots $. We can eliminate some cases by using \eqref{2leq} as before. If $p=17$ and $D_2\geq 27$, then \eqref{2leq} implies $2<(3/2)(17/16)+1/3+1/27+1/51<2$; if $p=17,D_2\geq17$, and $D_3\geq 81$, then \eqref{2leq} leads to $2<(3/2)(17/16)+1/3+1/17+1/81<2$. Similarly, if $p=13$, then we must have $D_2<39$; if $p=13$ and $D_2\geq 27$, then it forces $D_3<243$; if $p=11$, then $D_2<81$ or $D_3<243$. Therefore, we obtain 
	\begin{align}
		&\text{if $p=17$, then $(D_2=9)$ or ($D_2=17$ and  $D_3\in\{ 27,51\}$)}; \label{H1} \\
		& \text{if $p=13$, then $(D_2\in \{9,13\})$ or ($D_2=27$ and $D_3\in \{ 39,81,117,169\}$)}; \label{H2} \\
		&\text{if $p=11$, then $(D_2\in \{9,11,27,33\})$ or $(D_2=81$ and  $D_3 \in \{99,121\})$ or} \nonumber
		\\
		& \quad\quad\quad\quad\quad\quad\quad\text{$(D_2=99$ and $D_3=121$)}.  \label{H3} 
	\end{align}
	We divide our calculations according to the values of $D_2$ and $D_3$ listed in \eqref{H1}, \eqref{H2}, and \eqref{H3}.
	
	\indent\textbf{Case 6.1} $D_2=9$ (so $p$ can be any of $11$, $13$, or $17$). Since $D_1=3$ and $D_2=9$, the equation \eqref{F} holds. Substituting $p=11, 13, 17$ in \eqref{F}, we obtain, respectively
	\begin{align}
		&(17\cdot 3^{2\alpha -2}-11)(17\cdot 11^{2\beta}-27)=280-340d_3 \text{ (if $p=11$)}, \label{case611} \\
		&(15\cdot 3^{2\alpha-2}-13)(15\cdot 13^{2\beta}-27)=336-360d_3 \text{ (if $p=13$)}, \label{case612} \\
		&(11\cdot 3^{2\alpha -2}-17)(11\cdot 17^{2\beta}-27)=448-352d_3 \text{ (if $p=17$)}, \label{case613}
	\end{align}
	where $d_3$ in \eqref{case611} is a proper divisor of $3^{2\alpha}11^{2\beta}$, $d_3$ in \eqref{case612} is a proper divisor of $3^{2\alpha}13^{2\beta}$, and $d_3$ in \eqref{case613} is a proper divisor of $3^{2\alpha}17^{2\beta}$. Since $\alpha, \beta \geq 1$, the left-hand side of \eqref{case611} and \eqref{case612} are positive, while the right-hand side of \eqref{case611} and \eqref{case612} are negative. So \eqref{case611} and \eqref{case612} do not lead to a solution. For \eqref{case613}, we have $448-352d_3\leq 96$, which implies $\alpha=1$. Then \eqref{case613} reduces to $3\cdot 17^{2\beta}+13-16d_3=0$. Reducing this mod $3$ and mod $17$, we see that $d_3 \equiv 1 \Mod{3}$ and $d_3 \equiv 4 \Mod{17}$. Since $d_3 \mid 3^{2\alpha}17^{2\beta}$, $3\nmid d_3$, and $17\nmid d_3$, we obtain $d_3=1$, which contradicts the fact that $d_3\equiv 4 \Mod{17}$. Thus there is no solution in this case.
	
	\indent\textbf{Case 6.2} $D_2=p$ where $p\in \{11,13\}$. Similar to the calculation  for \eqref{F}, we have 
	\begin{align*}
		(3^{2\alpha+1}-1)(p^{2\beta+1}-1)&=2(p-1)\sigma (n)=2(p-1)(2n-d_1-d_2-d_3)
		\\&=2(p-1)(2\cdot 3^{2\alpha}p^{2\beta}-3^{2\alpha-1}p^{2\beta}-3^{2\alpha}p^{2\beta-1}-d_3). 
	\end{align*}
	Let $B_p=16p-p^2-6$. Following a straightforward algebraic manipulation and multiplying both sides by $B_p$, the above leads to
	\begin{equation}
		(B_p3^{2\alpha-1}-p^2)(B_p p^{2\beta-1}-9)=9p^2-B_p-2B_p(p-1)d_3. \label{H4}
	\end{equation}
	Substituting $p=11$ in \eqref{H4}, we obtain
	\begin{equation}
		(49\cdot3^{2\alpha-1}-121)(49\cdot 11^{2\beta-1}-9)=1040-980d_3. \label{case622}
	\end{equation}
	Since $\alpha, \beta \geq 1$, the left-hand side of \eqref{case622} is larger then $60$, while the right-hand side of \eqref{case622} is at most $60$, so \eqref{case622} does not give a solution. Next, substituting $p=13$ in \eqref{H4} and dividing both sides by $3$, we obtain
	\begin{equation}
		(33\cdot 3^{2\alpha-1}-169)(11\cdot13^{2\beta-1}-3)=496-264d_3. \label{case623}
	\end{equation}
	Since the right-hand side of \eqref{case623} is at most $232$, we obtain $\alpha=1$ and \eqref{case623} reduces to
	\begin{equation}
		35\cdot 13^{2\beta-1}-12d_3+13=0. \label{H5}
	\end{equation}
	Recall that $d_3 \mid n$ and $n=3^{2\alpha}p^{2\beta}=3^2\cdot 13^{2\beta}$. So $d_3=3^{a_3}13^{b_3}$ for some $a_3\in \{0,1,2\}$ and $b_3 \geq 0$. Reducing \eqref{H5} modulo $7$, we see that $2d_3\equiv 1 \Mod{7}$. If $a_3=0$, then $2d_3=2\cdot 13^{b_3}\equiv 2(-1)^{b_3}\equiv 2,-2 \not\equiv 1 \Mod{7}$. If $a_3=2$, then $2d_3=18\cdot13^{b_3}\equiv4(-1)^{b_3} \equiv4,-4 \not \equiv 1 \Mod{7}$. Therefore $a_3=1$ and \eqref{H5} becomes
	\begin{equation}
		35\cdot 13^{2\beta-1}-36\cdot 13^{b_3}+13=0. \label{case624}
	\end{equation}
	Suppose for a contradiction that $\beta \geq2$. Reducing \eqref{case624} modulo $13^2$, we obtain $36\cdot13^{b_3}\equiv 13 \Mod{13^2}$. If $b_3\geq 2$, then $36\cdot 13^{b_3}\equiv 0 \not\equiv 13\Mod{13^2}$. If $b_3=1$, then $36\cdot 13^{b_3}-13=35\cdot 13 \not\equiv 0 \Mod{13^2}$. If $b_3=0$, then $36\cdot13^{b_3}=36\not\equiv 13 \Mod{13^2}$. In any case, we reach a contradiction. Therefore $\beta=1$. Substituting $\beta=1$ in \eqref{case624}, we obtain $b_3=1$, and so $d_3=3^{a_3}13^{b_3}=39$. This leads to $n=3^{2\alpha}p^{2\beta}=3^2\cdot 13^2$ with the deficient divisors $d_1=n/D_1=3\cdot13^2=507$, $d_2=n/D_2=3^2\cdot13=117$ and $d_3=39$, which we already verified at the beginning of the proof that this is indeed a solution to our problem. The elimination for the other cases can be done in a similar way to the previous cases, so we give less details. Recall that $D_1=3$. The other cases are as follows:
	\begin{itemize}
\item[(i)] $p=17,D_2=17$, and $D_3\in \{27,51\}$ (this is the remaining case from \eqref{H1}).
\item[(ii)] $p=13,D_2=27$, and $D_3\in \{39,81,117,169\}$ (this is the remaining case from \eqref{H2}).
\item[(iii)] $p=11,D_2\in \{27,33\}$.
\item[(iv)] $p=11,D_2=81$, and $D_3\in \{99,121\}$.
\item[(v)] $p=11,D_2=99$, and $D_3=121$.
\end{itemize}	
	In (i),(ii),(iv), and (v), we know the values of $D_1$, $D_2$, $D_3$, and so we have the values of $d_1$, $d_2$, $d_3$. We start from the equality $\sigma(n)=2n-d_1-d_2-d_3$, perform the usual algebraic manipulation, and try to write the minimum nonnegative power of $3$ appearing among $d_1$, $d_2$, $d_3$ in terms of the other variables. We obtain the following results.
	For (i), we have $p=17$, $D_1=3$, $D_2=17$, and
	\begin{align*}
		&\text{if $D_3=27$, then $2\alpha \geq3$ and $3^{2\alpha-3}= \frac{289\cdot17^{2\beta-1}-1}{337\cdot17^{2\beta-1}-81} \in (0,1)$;}\\
		&\text{if $D_3=51$, then $3^{2\alpha-1}= \frac{289\cdot 17^{2\beta-1}-1}{9\cdot17^{2\beta-1}-9}\in (32,35)$,}
	\end{align*}
	which is a contradiction.
	For (ii), we have $p=13$, $D_1=3$, $D_2=27$, $2\alpha \geq3$, and
	\begin{align*}
		&\text{if $D_3=39$, then $3^{2\alpha-3}= \frac{169\cdot13^{2\beta-1}-1}{177\cdot13^{2\beta-1}-81} \in (0,1)$;}\\
		&\text{if $D_3=81$, then $2\alpha \geq 4$ and $3^{2\alpha-4}= \frac{13\cdot13^{2\beta}-1}{15\cdot 13^{2\beta}-243} \in (0,1)$;}\\
		&\text{if $D_3=117$, then $3^{2\alpha-3}= \frac{169\cdot 13^{2\beta-1}-1}{33\cdot13^{2\beta-1}-81} \in (5,7)$;}\\
		&\text{if $D_3=169$, then $3^{2\alpha-3}= \frac{2197\cdot13^{2\beta-2}-1}{141\cdot13^{3\beta-2}-81}\in (15,37)$.}
	\end{align*}
	The first three cases above give a contradiction. The last case implies that $$2197\cdot 13^{2\beta-2}-1=27(141\cdot13^{2\beta-2}-81),$$ which leads to $1610\cdot13^{2\beta-2}=2186$, which is impossible.
	For (iv), we have $p=11$, $D_1=3$, $D_2=81$, $2\alpha \geq4$, and 
	\begin{align*}
		&\text{if $D_3=99$, then $3^{2\alpha-4}= \frac{121\cdot11^{2\beta-1}-1}{103\cdot11^{2\beta-1}-243} \in (1,2)$;}\\
		&\text{if $D_3=121$, then $3^{2\alpha-4}= \frac{1331\cdot 11^{2\beta-2}-1}{773\cdot 11^{2\beta-2}-243} \in (1,3)$,}
	\end{align*}
	which is false.
	For (v), we have $p=11$, $D_1=3$, $D_2=99$, $D_3=121$, which leads to 
	$$3^{2\alpha-2}= \frac{1331\cdot 11^{2\beta-2}-1}{37\cdot11^{2\beta-2}-27} \in (35,37)\cup\{133\},$$ 
	which is not possible. We now consider (iii). We have $p=11$, $D_1=3$, $D_2\in\{27,33\}$. We know the values of $d_1$, $d_2$ but not $d_3$. We start with $\sigma(n)=2n-d_1-d_2-d_3$ and write $d_3$ in terms of the product of the other variables. Similar to the calculation for \eqref{F}, we obtain 
	\begin{align}
		&\text{if $D_2=27$, then $2\alpha\geq 3$ and $(3^{2\alpha-3}-1)(11^{2\beta+1}-81)=80-20d_3$}; \label{iii1}\\
		&\text{if $D_2=33$, then $(3^{2\alpha+1}-121)(11^{2\beta-1}-1)=120-20d_3$}. \label{iii2}
	\end{align}
	In \eqref{iii1}, $2\alpha$ is an even integer $\geq3$, so $2\alpha \geq4$, and thus the left-hand side of \eqref{iii1} is larger than $80$, while the right-hand side of \eqref{iii1} is less than $80$, which is a contradiction. Since the right-hand side of \eqref{iii2} is less than $120$, we see that $\alpha=1$ and \eqref{iii2} reduces to $47\cdot 11^{2\beta+1}-10d_3+13=0$. Reducing this modulo $11$, we see that $10d_3\equiv 2 \Mod{11}$, and therefore $d_3\equiv 9 \Mod{11}$. So $11\nmid d_3$. Since $d_3 \mid n$ and $n=3^{2\alpha}p^{2\beta}=3^2\cdot 11^{2\beta}$, we have $d_3=1,3,9$. Since $d_3\equiv 9 \Mod{11}$, $d_3=9$ only. Then $47\cdot 11^{2\beta+1}-90+13=0$. This leads to $47\cdot 11^{2\beta+1}=77$, which has no solution.
	
	\indent\textbf{Case 7} $p=7$. Then $\{D_1,D_2,D_3\} \subseteq \{3,7,9,21,\ldots\}$. If $D_1 \geq 7$ and $D_2\geq 21$, then \eqref{2leq} implies $2<(3/2)(7/6)+1/7+1/21+1/21 <2$ which is impossible. So ($D_1=3$) or ($D_1=7$ and $D_2=9$). If $D_1=3$, then $d_1=3^{2\alpha-1}7^{2\beta}$ and we have
	\begin{align*}
		0&=12\left(\sigma (n)-2n+d_1+d_2+d_3 \right) \\
		&= (3^{2\alpha+1}-1)(7^{2\beta+1}-1)-24n+12(d_1+d_2+d_3)\\
		&=3^{2\alpha}7^{2\beta}\left(21-3/7^{2\beta}-7/3^{2\alpha}-24\right)+1+12(d_1+d_2+d_3)\\
		&= 1+12d_1\left( 1+d_2/d_1+d_3/d_1\right)-3^{2\alpha}7^{2\beta}(3+3/7^{2\beta}+7/3^{2\alpha})\\
		&> 1+12d_1 -3^{2\alpha}7^{2\beta}(3+3/7^2+7/3^2)\\
		&> 12d_1-3^{2\alpha}7^{2\beta}(4)=0,
	\end{align*}
	which is a contradiction. So $D_1=7$ and $D_2=9$. We start with $\sigma (n)=2n-d_1-d_2-d_3$, substitute $d_1=3^{2\alpha}7^{2\beta-1}$, $d_2=3^{2\alpha-2}7^{2\beta}$, and do the usual algebraic manipulation to obtain
	\begin{equation}
		(3^{2\alpha-1}-49)(7^{2\beta-1}-9)=440-12d_3. \label{case7}
	\end{equation}
	If $\alpha \geq 3$ and $\beta \geq2$, then the left-hand side of \eqref{case7} is larger than $440$, while the right-hand side of \eqref{case7} is smaller than $440$. Therefore ($\alpha \in \{1,2\}$) or ($\alpha \geq3$ and $\beta =1$). Since $d_3 \mid n$ and $n=3^{2\alpha}7^{2\beta}$, $d_3=3^{a_3}7^{b_3}$ for some $a_3, b_3 \geq0.$ 
	
	\indent\textbf{Case 7.1} $\alpha \geq3$ and $\beta =1$. Then \eqref{case7} reduces to 
	\begin{equation}
		3^{2\alpha-1}+171=6\cdot3^{a_3}7^{b_3}. \label{K1}
	\end{equation}
	Since $3^{2\alpha-1}+171=3^2(3^{2\alpha-3}+19)$, we obtain $3^2 \mid\mid 6d_3$, which implies $a_3=1$. Dividing both sides of \eqref{K1} by $9$, we obtain $3^{2\alpha-3}+19=2\cdot7^{b_3}$. Reducing this modulo $3$, we have a contradiction.
	
	\textbf{Case 7.2} $\alpha \in \{1,2\}$. If $\alpha=2$, then \eqref{case7} leads to $d_3\equiv0 \Mod{11}$ which contradicts the fact that $d_3=3^{a_3}7^{b_3}$. So $\alpha=1$. Then $a_3\in \{0,1,2\}$ and \eqref{case7} reduces to $23\cdot 7^{2\beta-1}-6d_3+13=0$. From this, we see that $7\nmid d_3$. So $b_3=0,d_3=3^{a_3}$, and the above equation becomes $23\cdot 7^{2\beta-1}-6\cdot 3^{a_3}+13=0$. Substituting $a_3=0,1,2$, we obtain $23\cdot 7^{2\beta-1}=-7,5,41$, which is not possible. Hence there is no solution in this case.
	
	\indent\textbf{Case 8} $p=5$. Then the possible values of $D_1$, $D_2$, $D_3$ listed in an increasing order are $3, 5, 9, 15, 25, \ldots$. If $D_1\geq 25$, then \eqref{2leq} implies $2<(3/2)(5/4)+1/25+1/25+1/25<2$, which is false. Therefore $D_1\in \{3, 5, 9, 15\}$. It is possible to obtain bounds for $D_2$ and $D_3$ as in the other cases but the same method will lead to a longer calculation. In this case, it is better to get a bound only for $D_1$ and go back to $d_1$, $d_2$, $d_3$. Let $d_1=3^{a_1}5^{b_1}$, $d_2=3^{a_2}5^{b_2}$, and $d_3=3^{a_3}5^{b_3}$ where $a_i, b_i \geq 0$, and recall that $n>d_1>d_2>d_3\geq1$ and $d_1$, $d_2$, $d_3$ are the deficient divisors of $n=3^{2\alpha}5^{2\beta}$. In addition, from $\sigma (n)=2n-(d_1+d_2+d_3)$, we get
	\begin{align}
		(3^{2\alpha+1}-1)(5^{2\beta+1}-1)&=16\cdot 3^{2\alpha}5^{2\beta}-8(d_1+d_2+d_3)\notag\\
		&=16\cdot 3^{2\alpha}5^{2\beta}-8(3^{a_1}5^{b_1}+3^{a_2}5^{b_2}+3^{a_3}5^{b_3}). \label{A1}
	\end{align}
	From \eqref{A1}, we see that $8(d_1+d_2+d_3)=3^{2\alpha}5^{2\beta}+3^{2\alpha+1}+5^{2\beta+1}-1$, which implies
	\begin{equation}
		1< \frac{8}{3^{2\alpha}5^{2\beta}}(d_1+d_2+d_3)<1+ \frac{3}{5^2}+ \frac{5}{3^2}<2. \label{A2}
	\end{equation}
	Since $D_1 \in \{3,5,9,15\}$ and $d_1=n/D_1$, we see that
	\begin{equation}
		(a_1,b_1)=(2\alpha-1,2\beta), (2\alpha,2\beta-1),(2\alpha-2,2\beta), \ \text{or}\ (2\alpha-1,2\beta-1). \label{A3}
	\end{equation}
	Observe that $3^4\equiv 1 \Mod{5}, 5^2 \equiv 1 \Mod{3}$, and the exponents $4$ and $2$ are the smallest positive integers satisfying each congruence. From this, it is not difficult to verify that the left-hand side of \eqref{A1} satisfies
	\begin{align}
		&(3^{2\alpha+1}-1)(5^{2\beta+1}-1) \equiv 
		\begin{cases}
			3 \Mod{5}, &\text{if $\alpha$ is even; }\\
			4 \Mod{5}, &\text{if $\alpha$ is odd, }\label{LHSmod5}
		\end{cases}\\
		&(3^{2\alpha+1}-1)(5^{2\beta+1}-1) \equiv 2 \Mod{3}. \label{LHSmod3}
	\end{align}
	Since $5$ does not divide the left-hande side of \eqref{A1}, at least one of $d_1$, $d_2$, $d_3$ is not divisible by $5$, that is, at least one of $b_1$, $b_2$, $b_3$ is zero. By \eqref{A3}, we see that $b_1\neq 0$. Thus
	\begin{equation}
		b_1\neq0 \text{ and }\min \{b_2,b_3\}=0. \label{A4}
	\end{equation}
	Suppose for a contradiction that $a_1=a_2=a_3=0$. That is, $d_1=5^{b_1}$, $d_2=5^{b_2}$, $d_3=5^{b_3}$. Since $d_1>d_2>d_3$, we have $b_1>b_2>b_3$. So by \eqref{A4}, $b_3=0$ and $b_1>b_2>0$. Then the right-hand side of (\eqref{A1} is $\equiv2 \Mod{5}$ contradicting \eqref{LHSmod5}. So one of $a_1$, $a_2$, $a_3$ is not zero. By \eqref{LHSmod3} and \eqref{A1}, one of $d_1$, $d_2$, $d_3$ is not divisible by $3$, and so one of $a_1,a_2,a_3$ is zero. We conclude that 
	\begin{equation}
		\max \{a_1,a_2,a_3\} \geq 1 \text{ and } \min \{a_1,a_2,a_3\}=0. \label{A5}
	\end{equation}
	The right-hand side of \eqref{A1} is congruent to
	\begin{align}
		\begin{cases}
			(0+0+5^{b_3}) \Mod{3},& \text{if $a_1 \neq 0, a_2\neq 0$, and $a_3=0$};\\
			(0+5^{b_2}+0) \Mod{3},& \text{if $a_1 \neq 0, a_2= 0$, and $a_3\neq0$};\\
			(5^{b_1}+0+0) \Mod{3},& \text{if $a_1 = 0, a_2\neq 0$, and $a_3\neq 0$};\\
			(5^{b_1}+5^{b_2}+0) \Mod{3},& \text{if $a_1 = a_2=0$, and $a_3\neq 0$};\\
			(5^{b_1}+0+5^{b_3}) \Mod{3},& \text{if $a_1= a_3= 0$, and $a_2\neq 0$};\\
			(0+5^{b_2}+5^{b_3}) \Mod{3},& \text{if $a_2=a_3=0$, and $a_1 \neq 0$}.
		\end{cases} \label{A6}
	\end{align}
	By comparing \eqref{A1}, \eqref{LHSmod3} and \eqref{A6}, we obtain the parities of $b_1$, $b_2$, $b_3$ as follows. If $5^b \equiv 2 \Mod{3}$, then $b$ is odd. If $5^x+5^y \equiv 2\Mod{3}$, then $x$ and $y$ are even. For convenience, for each $i\in \{1,2,3\}$, if $b_i$ is odd, we write $b_i^\prime$ for $b_i$; if $b_i$ is even, then we replace $b_i$ by $b_i^{\prime\prime}$. Therefore, for each $i\in \{1,2,3\}$, $b_i^\prime$, $b_i^{\prime\prime}\geq 0$, $b_i^\prime =b_i$ is odd, and $b_i^{\prime\prime}=b_i$ is even, and there are six cases to consider as follows:
	
	\indent\textbf{Case 8.1} $d_1=3^{a_1}5^{b_1}$, $d_2=3^{a_2}5^{b_2}$, $d_3=5^{b_3^\prime}$, $a_1 \neq 0$, $ a_2\neq 0$, and $a_3=0$,
	
	\indent\textbf{Case 8.2} $d_1=3^{a_1}5^{b_1}$, $d_2=5^{b_2^\prime}$, $d_3=3^{a_3}5^{b_3}$, $a_1 \neq 0$, $a_2= 0$, and $a_3\neq0$,
	
	\indent\textbf{Case 8.3} $d_1=5^{b_1^\prime}$, $d_2=3^{a_2}5^{b_2}$, $d_3=3^{a_3}5^{b_3}$, $a_1= 0$, $a_2 \neq 0$, and $a_3\neq0$,
	
	\indent\textbf{Case 8.4} $d_1=5^{b_1^{\prime\prime}}$, $d_2=5^{b_2^{\prime\prime}}$, $d_3=3^{a_3}5^{b_3}$, $a_1= a_2 = 0$, and $a_3\neq0$,
	
	\indent\textbf{Case 8.5} $d_1=5^{b_1^{\prime\prime}}$, $d_2=3^{a_2}5^{b_2}$, $d_3=5^{b_3^{\prime\prime}}$, $a_1=a_3 = 0$, and $a_2\neq0$,
	
	\indent\textbf{Case 8.6} $d_1=3^{a_1}5^{b_1}$, $d_2=5^{b_2^{\prime\prime}}$, $d_3=5^{b_3^{\prime\prime}}$, $a_2=a_3 = 0$, and $a_1\neq0$.
	
	Some cases are shorter but we will begin with Case 8.1.
	
	\indent\textbf{Case 8.1} Since  $b_3^\prime \neq0$, we obtain by \eqref{A4} that $b_1 \neq0$ and $b_2=0$. By \eqref{A3}, there are 4 cases to consider. If $a_1=2\alpha-1$ and $b_1=2\beta$, then 
	$$
	8(d_1+d_2+d_3)/(3^{2\alpha}5^{2\beta}) = {8}\left( 3^{2\alpha-1}5^{2\beta}+3^{a_2}+5^{b_3^\prime} \right)/\left(3^{2\alpha}5^{2\beta}\right) > 8/3>2,
	$$
	which contradicts \eqref{A2}. Next, suppose that $a_1=2\alpha$ and $b_1=2\beta-1$. Since $3^{a_2}=d_2>d_3=5^{b_3^\prime} \geq5$, we obtain $a_2\geq2$. Thus
	\begin{align*}
		0&=8\left(\sigma (n)-2n+d_1+d_2+d_3\right)=8(d_1+d_2+d_3)-3^{2\alpha}5^{2\beta}-3^{2\alpha+1}-5^{2\beta+1}+1\\
		&>8(3^{2\alpha}5^{2\beta-1}+3^2+5)-3^{2\alpha}5^{2\beta}-3^{2\alpha+1}-5^{2\beta+1}\\
		&=(3^{2\alpha+1}-25)(5^{2\beta-1}-1)+87>0,
	\end{align*}
	which is false. Next, consider the case $(a_1,b_1)=(2\alpha-2,2\beta)$. Since $a_1\neq0$, $\alpha \geq2$. If $\beta \geq 2$, then \eqref{A2} implies that
	\begin{align*}
		1&< \frac{8}{3^{2\alpha}5^{2\beta}}(3^{2\alpha-2}5^{2\beta} +3^{a_2}+5^{b_3^\prime} ) \leq  \frac{8}{3^{2\alpha}5^{2\beta}} \left(3^{2\alpha-2}5^{2\beta}+3^{2\alpha}+5^{2\beta-1} \right) \\
		&= 8\left( \frac{1}{3^2} + \frac{1}{5^{2\beta}} +  \frac{1}{3^{2\alpha}\cdot 5}\right) \leq 8\left( \frac{1}{3^2}+ \frac{1}{5^4}+ \frac{1}{3^4 \cdot5} \right) <1,
	\end{align*} 
	which is a contradiction. So $\beta=1$. Then $d_3=5$.
	
	Starting with $0=8\left(\sigma(n)-2n+d_1+d_2+d_3\right)$, and then simplifying leads to $2\cdot3^{a_2}=13\cdot3^{2\alpha-2}+21$. Since $13\cdot3^{2\alpha-2}+21 > 2\cdot 3^{2\alpha-1}$, we obtain $a_2=2\alpha$. But then $21=2\cdot 3^{a_2}-13\cdot3^{2\alpha-2}=5\cdot 3^{2\alpha-2} \equiv0 \Mod{5}$, a contradiction. Next, we consider the last case: $(a_1,b_1)=(2\alpha-1,2\beta-1)$. If $\alpha \geq2$ or $\beta \geq2$, then \eqref{A2} implies 
	\begin{align*}
	1&< \frac{8}{3^{2\alpha}5^{2\beta}}\left(3^{2\alpha-1}5^{2\beta-1}+3^{a_2}+5^{b_3^\prime} \right) \\
	&\leq 8\left(  \frac{1}{15}+\max \left\{ \frac{1}{25}+ \frac{1}{3^4\cdot 5},  \frac{1}{5^4}+ \frac{1}{3^2\cdot5} \right\} \right) <1,
	\end{align*}
	which is impossible. So $\alpha=1=\beta$. Then $a_1=1=b_1$. Since $15=d_1>3^{a_2}=d_2>d_3=5^{b_3^\prime}=5$, we have $d_2=9$. Now it is easy to verify that $\sigma(n)-2n+d_1+d_2+d_3=-18\neq0$. So there is no solution in this case.
	
	\indent\textbf{Case 8.2} Since $b_2=b_2^\prime \neq 0 $, we obtain by \eqref{A4} that $b_3=0$. Similar to Case 8.1, we divide our calculation into four cases according to the values of $a_1$ and $b_1$ as given in \eqref{A3}. If $(a_1,b_1)=(2\alpha-1, 2\beta)$, then $8(d_1+d_2+d_3)/(3^{2\alpha}5^{2\beta}) > 8d_1/(3^{2\alpha}5^{2\beta}) > 8/3 >2$, contradicting \eqref{A2}.
	If $(a_1,b_1)=(2\alpha,2\beta-1)$, then $d_2\geq5, d_3\geq3$, and
	\begin{align*}
		0&=8\left(\sigma (n)-2n+d_1+d_2+d_3\right) = 8(d_1+d_2+d_3)-3^{2\alpha}5^{2\beta}-3^{2\alpha+1}-5^{2\beta+1}+1\\
		&\geq3^{2\alpha+1}5^{2\beta-1}-3^{2\alpha+1}-5^{2\beta+1}+65\\
		&= (3^{2\alpha+1}-25)(5^{2\beta-1}-1)+40>0,
	\end{align*}
	which is not possible. Suppose $(a_1,b_1)=(2\alpha-2,2\beta)$. Since $a_1 \neq0$, we have $\alpha \geq2$. If $\beta \geq 2$, then \eqref{A2} implies
	\begin{align*}
		1&< \frac{8}{3^{2\alpha}5^{2\beta}}( 3^{2\alpha-2}5^{2\beta}+5^{b_2^\prime} + 3^{a_3})\\
		&\leq  \frac{8}{3^{2\alpha}5^{2\beta}}\left( 3^{2\alpha-2}5^{2\beta}+5^{2\beta-1}+3^{2\alpha} \right) \leq 8 \left( \frac{1}{9}+ \frac{1}{3^4\cdot5}+ \frac{1}{5^4}\right)<1,
	\end{align*}
	which is false. So $\beta=1$. Then $d_2=5$ and $d_3=3$. 
	
	Starting from $8\left(\sigma(n)-2n+d_1+d_2+d_3\right)=0$ and then simplifying leads to $13\cdot3^{2\alpha-2}+15=0$, which is impossible. The last case of \eqref{A3} is $(a_1,b_1)=(2\alpha-1,2\beta-1)$. If $\alpha \geq 2$ or $\beta \geq 2$, then \eqref{A2} implies 
	\begin{align*}
		1&<  \frac{8}{3^{2\alpha}5^{2\beta}} \left( 3^{2\alpha-1} 5^{2\beta-1}+5^{b_2^\prime}+3^{a_3}\right) \\
		&\leq  8\left(  \frac{1}{15}+ \frac{1}{3^{2\alpha}\cdot 5}+ \frac{1}{5^{2\beta}}\right) \\
		&\leq  8\left(  \frac{1}{15}+\max \left\{  \frac{1}{3^4\cdot 5}+ \frac{1}{5^2}, \frac{1}{3^2\cdot 5}+ \frac{1}{5^4}\right\}\right) <1,
	\end{align*}
	which is not ture. Thus $\alpha=\beta=1$. So $a_1=b_1=1$, $d_2=5$, and $d_3=3$. Now it is easy to verify that $\sigma(n)-2n+d_1+d_2+d_3=-24\neq 0$. So there is no solution in this case.
	
	\indent\textbf{Case 8.3} By \eqref{A2}, we obtain $1<8(3d_1)/(3^{2\alpha}5^{2\beta}) \leq 24\cdot5^{2\beta-1}/(3^{2\beta}5^{2\beta}) \leq 24/45 <1 $, a contradiction.
	
	\indent\textbf{Case 8.4} By \eqref{A2}, we obtain 
	\begin{align*}
	1&< \frac{8}{3^{2\alpha}5^{2\beta}}(d_1+2d_2)= \frac{8}{3^{2\alpha}5^{2\beta}}(5^{b_1^{\prime\prime}}+2\cdot 5^{b_2^{\prime\prime}}) \\
	&\leq  \frac{8}{3^{2\alpha}5^{2\beta}}(5^{2\beta}+2\cdot 5^{2\beta-2})\leq 8\left(  \frac{1}{9}+ \frac{2}{9\cdot25} \right) <1,
	\end{align*}
	which is not possible.
	
	\indent\textbf{Case 8.5} If $\alpha \geq2$, then \eqref{A2} implies that 
	$$1< \frac{8}{3^{2\alpha}5^{2\beta}}(2d_1+d_3) \leq \frac{8}{3^{2\alpha}5^{2\beta}}(2\cdot 5^{2\beta}+5^{2\beta-2}) \leq 8\left(  \frac{2}{3^4}+ \frac{1}{3^4\cdot 5^2} \right) <1,$$ 
	which is false. Therefore $\alpha=1$. Then the left-hand side of \eqref{A1} is $\equiv 4 \Mod{5}$, while the right-hand side is $\equiv 2 (d_1+d_2+d_3) \equiv 2(3^{a_2}5^{b_2}+5^{b_3^{\prime\prime}}) \Mod{5}$. By \eqref{A4}, $b_2=0$ or $b_3=0$. If $b_2=0$ and $b_3\neq 0$, then $5^2 \leq d_3 <d_2 =3^{a_2}$, and so $a_2 \geq3$ contradicting the fact that $d_2 \mid n$ and $n=3^{2\alpha}5^{2\beta}=3^2 \cdot 5^{2\beta}$. If $b_2 \neq 0$ and $b_3=0$, then $2(3^{a_2}5^{b_2}+5^{b_3^{\prime\prime}})\equiv 2 \Mod{5}$, which is not the case. Since $\alpha =1$, $a_2 \in \{1,2\}$. So if $b_2=b_3=0$, then $2(3^{a_2}5^{b_2}+5^{b_3^{\prime\prime}}) \equiv3,0 \Mod{5}$, which is not true. So there is no solution in this case.
	
	\indent\textbf{Case 8.6} Since $5^{b_2^{\prime\prime}}=d_2>d_3\geq1$, we have $b_2^{\prime\prime} \neq0$. By \eqref{A4}, we see that $b_3=0$. Then the right-hand side of \eqref{A1} is $\equiv 2(3^{a_1}5^{b_1}+5^{b_2^{\prime\prime}}+5^{b_3^{\prime\prime}})\equiv 2 \Mod{5}$ contradicting \eqref{LHSmod5}. 
	
	This completes the proof of this theorem.
\end{proof}

\end{document}